\documentclass[reqno,a4paper]{amsart}
\usepackage{amsmath,amsthm,amsfonts,amssymb,xcolor,hyperref}
\usepackage[all]{xy}
\usepackage{stmaryrd}
\usepackage{mathtools}
\usepackage{tikz}
\usepackage{mathabx}
\usepackage{calrsfs}
\usetikzlibrary{cd, arrows,arrows.meta}

\newtheorem{lemma}{Lemma}[section]
\newtheorem{proposition}[lemma]{Proposition}
\newtheorem{theorem}[lemma]{Theorem}
\newtheorem{corollary}[lemma]{Corollary}

\theoremstyle{remark}
\newtheorem{example}[lemma]{Example}

\newtheorem{remark}[lemma]{Remark}

\theoremstyle{definition}
\newtheorem{definition}[lemma]{Definition}

\newenvironment{tfae}
{
	\begin{enumerate}}
	{\end{enumerate}}

\newcommand{\Dhead}{-Triangle[open]}
\newcommand{\Dtail}{Triangle[open, reversed]->}

\newcommand{\defn}{\textbf}
\newcommand{\psj}{\curlyvee}
\newcommand{\noproof}{\hfill \qed}

\newcommand{\ua}{{\uparrow}}

\DeclareMathOperator{\Sub}{Sub}

\newcommand{\FWACC}{{\rm (FWACC)}}
\newcommand{\LACC}{{\rm (LACC)}}

\newcommand{\NH}{{\rm (NH)}}
\newcommand{\SH}{{\rm (SH)}}
\newcommand{\SSH}{{\rm (SSH)}}
\newcommand{\W}{{\rm (W)}}

\newcommand{\AC}{{\rm (AC)}}

\newcommand{\gnab}{\text{\rm \textexclamdown}} 
\newcommand{\oleq}{\subseteq} 
\newcommand{\normal}{\trianglelefteq} 
\newcommand{\omeet}{\cap} 
\newcommand{\bigojoin}{\bigcup} 
\newcommand{\ojoin}{\cup}  
\newcommand{\iimplies}{\shortrightarrow}

\newcommand{\HSLat}{\ensuremath{\mathsf{HSLat}}}
\newcommand{\Grp}{\ensuremath{\mathsf{Grp}}}

\DeclareMathOperator{\Ker}{Ker}

\newcommand{\ELHSLat}{\ensuremath{\mathsf{ELHSLat}}}
\DeclareMathOperator{\Huq}{Huq}

\newcommand{\C}{\ensuremath{\mathcal{C}}}
\newcommand{\Pt}{\ensuremath{\mathsf{Pt}}}

\renewcommand{\leq}{\leqslant}
\renewcommand{\nleq}{\nleqslant}

\renewcommand{\geq}{\geqslant}

\begin{document}
\title[Observations on \(\ELHSLat\)]{Observations on the variety of\\ equationally linear Heyting semilattices}
\author[García-Martínez]{X.~García-Martínez}
\author[Gray]{J.~R.~A.~Gray}
\author[Hoefnagel]{M.~A.~Hoefnagel}
\author[Van~der Linden]{T.~Van~der Linden}
\author[Vienne]{C.~Vienne}

\address[Xabier García-Martínez]{CITMAga \& Universidade de Santiago de Compostela, Departamento de Mate\-máticas, R\'ua Lope G\'omez de Marzoa s/n, 15782 Santiago de Compostela, Spain}
\email{xabier.garcia@usc.gal}
\address[James R.\ A.\ Gray, Michael A.\ Hoefnagel]{Mathematics Division, Department of Mathematical Sciences, Stellenbosch University, Private Bag X1, 7602 Matieland, South Africa}
\email{jamesgray@sun.ac.za}
\email{mhoefnagel@sun.ac.za}
\address[Tim Van~der Linden, Corentin Vienne]{Institut de
	Recherche en Math\'ematique et Physique, Universit\'e catholique
	de Louvain, che\-min du cyclotron~2 bte~L7.01.02, B--1348
	Louvain-la-Neuve, Belgium}
\email{tim.vanderlinden@uclouvain.be}
\email{corentin.vienne@uclouvain.be}
\address[Tim Van der Linden]{Mathematics \& Data Science, Vrije Universiteit Brussel, Pleinlaan 2, B--1050 Brussel, Belgium}
\email{tim.van.der.linden@vub.be}
\address[Corentin Vienne]{Dipartimento di Matematica ``Federigo Enriques'', Universit\`a degli Studi di Milano, Via Saldini 50, 20133 Milano, Italy}

\thanks{The first author is supported by grant PID2024-155502NB-I00 supported by MICIU/AEI /10.13039/501100011033 and FEDER, UE, and by Xunta de Galicia through the Competitive Reference Groups (GRC), ED431C	2023/31. The fourth author is a Senior Research Associate of the Fonds de la Recherche Scientifique--FNRS. The fifth author is supported by means of a \emph{bourse excellence mobilité} of WBI.WORLD}

\begin{abstract}
	In previous work, we analysed a number of categorical properties, of interest in the context of Janelidze--Márki--Tholen semi-abelian categories, for the variety \(\HSLat\) of Heyting semilattices. In this paper, we focus on the subvariety \(\ELHSLat\) of equationally linear Heyting semilattices. Our main objective is to show that, unlike \(\HSLat\), this category is algebraically coherent. We furthermore prove that it is neither locally algebraically cartesian closed nor cosmash associative.
\end{abstract}

\keywords{Equationally linear Heyting semilattice, arithmetical semi-abelian category, action accessibility, algebraic coherence, normality of unions, Smith is Huq condition, transitivity of normality, Higgins commutator, centraliser, normaliser}

\maketitle{}

\section{Introduction}

In a recent paper by the same authors~\cite{HSLat}, it was shown that the category \(\HSLat\) of Heyting semilattices satisfies several categorical-algebraic properties: \emph{\(n\)-normality of Higgins commutators} as a consequence of arithmeticity, \emph{transitivity of normality} and its corollaries, \emph{algebraic cartesian closedness}, and the condition \SSH{}. It was also shown there that this category is not \emph{algebraically coherent} (and thus is not \emph{locally algebraically cartesian closed}), \emph{action accessible}, or \emph{cosmash associative}, and that it fails to satisfy either \emph{normality of unions} or the condition~\W{}. For all these notions, we refer to~\cite{HSLat} and the literature cited therein.

The main objective of this paper is to study the subvariety \(\ELHSLat\) of \emph{equationally linear} Heyting semilattices, determined by the additional identity
\[
	((x \iimplies y) \iimplies z) \wedge ((y \iimplies x) \iimplies z) = z.
\]
From the results of~\cite{HSLat}, we may immediately transfer several properties to this subvariety, since the original proofs remain valid.

\begin{theorem}
	The category \(\ELHSLat\) is arithmetical, and therefore satisfies \(n\)-normality of Higgins commutators. It also satisfies transitivity of normality, algebraic cartesian closedness, and the condition \SSH{}. However, \(\ELHSLat\) does not satisfy normality of unions, nor is it action accessible. \noproof
\end{theorem}

In the present paper, we prove that, unlike \(\HSLat\), the category \(\ELHSLat\) \emph{does} satisfy algebraic coherence, as well as the condition \W{}. On the other hand, we show that \(\ELHSLat\) is \emph{not} locally algebraically cartesian closed, nor is it cosmash associative.

The article is organised as follows. In Section~\ref{section preliminaries}, we recall the necessary background on (equationally linear) Heyting semilattices and review the categorical tools used throughout the paper. Section~\ref{section acc and w} is devoted to the study of centralisers and the condition~\W{} in \(\ELHSLat\). In Section~\ref{section ac}, we prove the main result of the paper, namely that \(\ELHSLat\) is algebraically coherent. In Section~\ref{section lacc}, we show conversely that \(\ELHSLat\) fails to be locally algebraically cartesian closed, and in Section~\ref{section cosmash} that it fails to be cosmash associative. Finally, in Appendix~\ref{Appendix}, we explain how the identities of Lemma~\ref{Equations} are established within the subvariety of equationally linear Heyting semilattices.

\section{Preliminaries}\label{section preliminaries}

In this section, we recall some facts about (equationally linear) Heyting semilattices, as well as the categorical tools needed in this work. For further details we refer to~\cite{HSLat}, where a more extensive presentation is given, along with additional references.

\subsection*{Heyting semilattices}

A~Heyting semilattice \((H,1,\wedge ,\iimplies)\) consists of a set \(H\), a constant \(1\in H\) and two operations \(\wedge\), \(\iimplies \colon H \times H \rightarrow H\) named \defn{meet} and \defn{implication} satisfying the equations
\begin{align*}
	1\wedge x           & = x,                    & x\iimplies x             & =1,                                     \\
	x\wedge x           & =x,                     & x\wedge ( x \iimplies y) & = x \wedge y,                           \\
	x\wedge y           & =y\wedge x,             & y \wedge (x \iimplies y) & =y,                                     \\
	x\wedge (y\wedge z) & = (x \wedge y)\wedge z, & x \iimplies (y \wedge z) & = (x\iimplies y) \wedge (x\iimplies z).
\end{align*}
The identities on the left express that \(H\) is a partially ordered set with a top element~\(1\). The identities on the right express the implication to be right adjoint to meet. In fact, it is equivalent to think about Heyting semilattices as \emph{cartesian closed posets}. One can see that a poset \((H,\leq)\) is a Heyting semilattice precisely when there exists a (necessarily unique) operation \(\iimplies \colon H \times H \rightarrow H\) such that
\begin{align*}
	x \wedge y \leq z \quad \text{if and only if} \quad x \leq y\iimplies z.
\end{align*}
This exhibits the left adjointness of product (\(\wedge\)) to internal hom (\(\iimplies\)). This means that for any poset there is at most one Heyting semilattice structure on it. Of course, morphisms between Heyting semilattices are defined in the expected way and we write \(\HSLat\) for the category of Heyting semilattices. As a variety in the sense of Universal Algebra, it is \emph{semi-abelian}~\cite{Janelidze-Marki-Tholen}. This was proven by Johnstone in~\cite{Johnstone:Heyting}, using the following characterisation due to Bourn and Jan\-elidze~\cite{Bourn-Janelidze}:

\begin{theorem}\label{BouJa Charact}
	A variety is semi-abelian if and only if the corresponding algebraic theory contains exactly one constant \(0\) and for some natural number \(n\) we have
	\begin{itemize}
		\item \(n\) binary operations \(\alpha_i\) such that \(\alpha_i(x,x) =0\) for all \(i=1, \dots , n\),
		\item an \((n+1)\)-ary operation \(\Theta\) such that \(\Theta(\alpha_1(x,y), \dots , \alpha_n(x,y),y)= x\).\noproof
	\end{itemize}
\end{theorem}

Johnstone showed that taking the constant \(0\) to be \(1\) and the operations
\begin{align}\label{protomodular terms for HSL}
	\alpha_1(x,y)= x \iimplies y, \;
	\alpha_2(x,y)= ((x\iimplies y)\iimplies y)\iimplies x, \;
	\Theta(x,y,z) = (x \iimplies z) \wedge y,
\end{align}
the conditions of Theorem~\ref{BouJa Charact} are satisfied.

We say that a subobject is \defn{normal} when it is the kernel of some morphism.
To avoid confusion with the intrinsic order of a Heyting semilattice, we adopt the following conventions.
For an object \(H\), the order on its class \(\Sub(H)\) of subobjects is denoted by ``\(\oleq\)'',
with meets and joins written as ``\(\omeet\)'' and ``\(\ojoin\)'', respectively.
Each monomorphism \(x \colon X \rightarrowtail H\) determines a subobject, which we identify with~\(x\) or, equivalently, write as \(X \oleq H\).
Finally, we write \(K \normal H\) when a subobject \(K \oleq H \) is normal.

Before characterising normal subobjects in \(\HSLat\), we recall the notion of a filter.

\begin{definition}
	A non-empty subset \(K\) of a Heyting semilattice \(H\) is a \defn{filter} of~\(H\) when:
	\begin{enumerate}
		\item if \(k, k' \in K\), then \(k \wedge k' \in K\);
		\item if \(k \in K\) and \(h \in H\) with \(h \geq k\), then \(h \in K\).
	\end{enumerate}
\end{definition}

Nemitz proved in~\cite{Nemitz} that:

\begin{lemma}\label{Lemma Normal is Filter}
	A non-empty subset \(K\) of a Heyting semilattice \(H\) is the kernel of some morphism \(f \colon H \to H'\) of Heyting semilattices if and only if \(K\) is a filter. In categorical terminology, filters correspond to \defn{normal subobjects} in the category \(\HSLat\). \noproof
\end{lemma}

As observed in~\cite{HSLat}, interpreting normal monomorphisms as filters makes it clear that in \(\HSLat\), the composite of two normal monomorphisms is again normal. This motivated the study of what we called \emph{transitivity of normality} in that same paper.

\subsection*{Equationally linear Heyting semilattices}

Following~\cite{MR0245486}, the term
\[
	x \psj y \coloneq ((x \iimplies y) \iimplies y) \wedge ((y \iimplies x) \iimplies x)
\]
is called the \defn{pseudo-join} of \(x\) and \(y\). It satisfies the following properties:
\begin{enumerate}
	\item \(x \leq x \psj y\) and \(y \leq x \psj y\);
	\item \(x \leq y\) if and only if \(x \psj y = y\);
	\item \(x \psj x = x\), and \(x \psj y = y \psj x\);
	\item \(x \psj (x \wedge y) = x = x \wedge (x \psj y)\).
\end{enumerate}
A perhaps surprising, yet easy to verify identity is the distributivity of implication over the pseudo-join:
\[
	x \iimplies (y \psj z) = (x \iimplies y) \psj (x \iimplies z).
\]

It is worth noting that Heyting semilattices need not have their pseudo-join coincide with the join, as illustrated below.

\begin{example}\label{example non semi-boolean}
	Let \(H\) be the Heyting semilattice with Hasse diagram
	
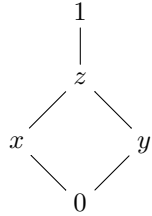
\[
		\xymatrix@!=0ex{
		& 1 &\\
		& z\ar@{-}[u] &\\
		x\ar@{-}[ur] & & y \ar@{-}[ul]\\
		& 0\ar@{-}[ul]\ar@{-}[ur] &
		}
	\]
	Then the pseudo-join \(\psj\) is not the join, since \(x \psj y = 1\).
\end{example}

Following Nemitz, a Heyting semilattice will be called \defn{semi-Boolean} when its pseudo-join coincides with its join. In~\cite{MR0245486}, he established the following characterisation.

\begin{proposition}\label{proposition equivalences semi-boolean}
	For a Heyting semilattice \(H\), the following conditions are equivalent:
	\begin{enumerate}
		\item \(H\) is semi-Boolean;
		\item \(x \psj (y \psj z) = (x \psj y) \psj z\);
		\item \(x \leq y\) implies \((z \psj x) \leq (z \psj y)\);
		\item \(\psj\) distributes over \(\wedge\);
		\item \(\wedge\) distributes over \(\psj\);
		\item \((x \psj y) \iimplies z = (x \iimplies z) \wedge (y \iimplies z)\).\noproof
	\end{enumerate}
\end{proposition}

\begin{example}
	Any linearly ordered set with a top element can be endowed with a semi-Boolean Heyting semilattice structure. Indeed, given a linearly ordered set \(S\) with top element \(1\), the meet of two elements is their minimum and the join is their maximum. The implication is defined by
	\[
		x \iimplies y =
		\begin{cases}
			1 & \text{if } x \leq y, \\
			y & \text{otherwise},
		\end{cases}
	\]
	which turns \(S\) into a semi-Boolean Heyting semilattice.
\end{example}

In~\cite[Theorem~3.3]{NeWh1}, Nemitz--Whaley proved that the semi-Boolean Heyting semilattices form the subvariety of \(\HSLat\) generated by the linearly ordered ones; that is, the smallest class of Heyting semilattices containing all linearly ordered Heyting semilattices and closed under products, subobjects, and quotients. In~\cite{Kohler}, Köhler characterised this subvariety by the identity
\begin{align}\label{kohler's identity}
	[((x \iimplies y) \iimplies z) \wedge ((y \iimplies x) \iimplies z)] \iimplies z = 1.
\end{align}

In the unpublished manuscript~\cite{FreydHeyting}, Freyd introduced the notion of an \defn{equationally linear} Heyting semilattice, defined as one satisfying:
\[
	((x \iimplies y) \iimplies z) \wedge ((y \iimplies x) \iimplies z) = z,
\]
which is easily seen to be equivalent to Köhler's identity~(\ref{kohler's identity}). In what follows, we adopt Freyd's terminology and denote the corresponding variety by \(\ELHSLat\). Moreover, since in this variety the join and the pseudo-join coincide, we write \(\vee\) instead of \(\psj\) when working in \(\ELHSLat\).

\subsection*{Categorical background}

As mentioned above, the category of (equationally linear) Heyting semilattices is semi-abelian. For the sake of simplicity, in what follows, we assume this condition always holds, although in some cases this assumption can be weakened. The reader is not required to be familiar with semi-abelian categories. This notion was introduced in~\cite{Janelidze-Marki-Tholen}, and we recommend the reference book~\cite{Borceux-Bourn} as a comprehensive guide to the subject.

In the context of a semi-abelian category, there is a notion of commutator for subobjects, essentially due to Huq~\cite{Huq,Borceux-Bourn}.

\begin{definition}
	Given two subobjects \(x \colon X \rightarrowtail H\) and \(y\colon Y \rightarrowtail H\), a \defn{cooperator} \(\varphi\) on the pair \((X,Y)\) with respect to \(H\) is a morphism \(\varphi \colon X\times Y \rightarrow H\) such that the following diagram commutes.
	\begin{center}
		\begin{tikzcd}[row sep = large]
			X \arrow[r, "{(1_X,0)}"] \arrow[rd, "x"', tail] & X \times Y \arrow[d, "\varphi" description] & Y \arrow[l, "{(0,1_Y)}"'] \arrow[ld, "y", tail] \\
			& H
		\end{tikzcd}
	\end{center}
	When such a cooperator exists, we say that \(X\) and \(Y\) \defn{cooperate} or \defn{Huq-commute} in \(H\).
\end{definition}

One can check that the previous definition behaves as expected in the category of groups. Indeed, two subgroups \(X\) and \(Y\) cooperate if and only if \(xy = yx\) for all \(x \in X\) and \(y \in Y\). In the case of Heyting semilattices, we have the following characterisation:

\begin{proposition}[\cite{HSLat}]\label{proposition commute in hslat}
	Let \(X\), \(Y \oleq H\) be two subobjects of a Heyting semilattice~\(H\). The following conditions are equivalent:
	\begin{tfae}
		\item \(X\) and \(Y\) Huq-commute in \(H\);
		\item for all \(x \in X\) and \(y \in Y\),
		\begin{align*}
			(x \iimplies y) = y \qquad \text{and} \qquad (y \iimplies x) = x;
		\end{align*}
		\item for all \(x \in X\) and \(y \in Y\), \(x \psj y = 1\).\noproof
	\end{tfae}
\end{proposition}

\begin{definition}
	The \defn{Huq-commutator} \([X,Y]^{\Huq}\) of \(X\) and \(Y\) in \(H\) is the smallest normal subobject \(N\) of \(H\) such that \(X\) and~\(Y\) cooperate in \({H}/{N}\).
\end{definition}
Given two subobjects \(X\) and \(Y\) of \(H\), we construct the diagram
\begin{center}

	\begin{tikzcd}[row sep = large]
		& X \arrow[dl,"{(1_X,0)}"'] \arrow[d, dotted] \arrow[dr,tail,"x"] & \\

		X\times Y \arrow[r,dotted] & Q & H \arrow[l, dotted,"q"']\\
		& Y \arrow[u,dotted] \arrow[ur,tail,"y"'] \arrow[ul,"{(0,1_Y)}"]
	\end{tikzcd}
\end{center}
where \(Q\) is the colimit of the outer square.

This object \([X,Y]^{\Huq}\) is the kernel of \(q\colon H \rightarrow Q\). The following helps to understand the previous definitions.

\begin{proposition}[\cite{Borceux-Bourn}]
	Given two subobjects \(X\) and \(Y\) in \(H\), they admit a cooperator if and only if \([X,Y]^{\Huq}=0\).\qed
\end{proposition}

The \textit{Higgins commutator} provides another approach to commutators of subobjects in the semi-abelian context. Originally defined for \(\Omega\)-groups by Higgins~\cite{Higgins} and later extended to a categorical setting by Mantovani and Metere~\cite{MM-NC}, it adopts a more universal--algebraic viewpoint.

We start by considering the canonical morphism
\[
	\Sigma_{X,Y}\coloneq
	\begin{psmallmatrix}
		1_X & 0 \\
		0 & 1_Y
	\end{psmallmatrix}
	=(\langle1_X,0\rangle,\langle0,1_Y\rangle)\colon X+Y\to X\times Y
\]
for any two objects \(X\) and \(Y\).
The \defn{binary cosmash product} \(X\diamond Y\) is defined as the kernel of \(\Sigma_{X,Y}\).
In the category of groups, \(X\diamond Y\) is generated by the \emph{words} in the coproduct \(X+Y\) of the form \(xyx^{-1}y^{-1}\) with \(x\in X\) and \(y\in Y\).

\begin{definition}[\cite{MM-NC}]\label{definition higgins commutator}
	Let \(x\colon X\rightarrowtail H\) and \(y\colon Y\rightarrowtail H\) be subobjects.
	They induce a morphism \(\langle x,y\rangle\colon X+Y\to H\).
	The \defn{Higgins commutator} \([X,Y]\) of \(X\) and \(Y\) in~\(H\) is defined as the regular image of \(X\diamond Y\) under \(\langle x,y\rangle\).

	\begin{center}
		\begin{tikzcd}[column sep = large]
			0 \arrow[r]
			& X\diamond Y \arrow[r, \Dtail] \arrow[d, \Dhead, dotted]
			& X+Y \arrow[r, "\Sigma_{X,Y}"] \arrow[d, "{\langle x, y\rangle}"]
			& X\times Y \\
			& {[X,Y]} \arrow[r, tail, dotted]
			& H
		\end{tikzcd}
	\end{center}
\end{definition}

A key relationship, proved in~\cite{MM-NC}, is that for any two subobjects \(X\), \(Y\) of an object \(H\), the Huq commutator \([X,Y]^{\Huq}\) is the normal closure of the Higgins commutator \([X,Y]\).
This shows that the two notions do differ. Nevertheless, one always has
\[
	[X,Y]^{\Huq}=0
	\quad\Leftrightarrow\quad
	[X,Y]=0.
\]

Since both \( \HSLat \) and \( \ELHSLat \) are \emph{arithmetical} categories (see~\cite{Pedicchio2} and \cite{Pixley}), we obtain the following as a direct corollary of Theorem~3.7 in~\cite{HSLat}:

\begin{proposition}\label{proposition commutator is intersection}
	Let \(X\), \(Y \oleq H\) be two subobjects of an (equationally linear) Heyting semilattice \(H\).
	Then the Higgins commutator of \(X\) and \(Y\) in \(H\) is given by
	\[
		[X,Y] = X^{Y} \omeet Y^{X},
	\]
	where \(X^{Y}\) denotes the normal closure of \(X\) inside \(X \ojoin Y\) (and similarly for \(Y^{X}\)).\noproof
\end{proposition}

We now recall three central notions in semi-abelian categories that will appear throughout this paper: \emph{split extensions}, \emph{points}, and \emph{kernel functors}.

Let \(X\) and \(B\) be objects in a semi-abelian category \(\C\).
A \defn{split extension} of \(B\) by \(X\) is a diagram
\begin{equation}\label{eq:split_ext}
	\begin{tikzcd}
		X \arrow[r, "\kappa"]
		& A \arrow[r, shift left, "\alpha"]
		& B \arrow[l, shift left, "\beta"]
	\end{tikzcd}
\end{equation}
in \(\C\) such that \(\alpha \circ \beta = 1_B\) and \((X,\kappa)\) is the kernel of \(\alpha\).
Morphisms of split extensions are morphisms of extensions commuting with the given sections.

We write \(\Pt_B(\C)\) for the category of \defn{points} over \(B\), whose objects are triples \((A,\alpha,\beta)\) with \(\alpha\colon A \to B\) a split epimorphism and \(\beta\) a chosen section. Morphisms are defined in the evident way.
Given a morphism \(p\colon E \to B\) in \(\C\), the \defn{change-of-base functor} along \(p\) is the functor
\[
	p^* \colon \Pt_B(\C) \longrightarrow \Pt_E(\C)
\]
sending a point \((A,\alpha,\beta)\) to the pullback of \(\alpha\) along \(p\).

In particular, the functor
\[
	\Ker_B \colon \Pt_B(\C) \longrightarrow \C,
\]
which assigns to each point \((A,\alpha,\beta)\) the kernel of \(\alpha\), plays a central role in the literature and in the present work.
We refer to it as the \defn{kernel functor}.

\begin{remark}\label{Remark Split Extension}
	In accordance with Theorem~\ref{BouJa Charact}, for any split extension of Heyting semilattices as in~\eqref{eq:split_ext}, every element \(a \in A\) can be written as
	\[
		a
		= \bigl( (a \iimplies \beta\alpha(a)) \iimplies \beta\alpha(a) \bigr)
		\wedge
		\bigl( ((a \iimplies \beta\alpha(a)) \iimplies \beta\alpha(a)) \iimplies a \bigr),
	\]
	where both
	\(
	a \iimplies \beta\alpha(a)
	\)
	and
	\(
	((a \iimplies \beta\alpha(a)) \iimplies \beta\alpha(a)) \iimplies a
	\)
	belong to \(X\), while \(\beta\alpha(a)\) is in \(B\).
	Equivalently, we may write \(a = (x_1 \iimplies b) \wedge x_2\) for some \(x_1\), \(x_2 \in X\) and \(b \in B\).
	Viewing \(X\) and \(B\) as subobjects of \(A\), we thus see that \(A\) is generated by \(X\) and \(B\).
\end{remark}

\section{Centralisers and Commutators}\label{section acc and w}

In~\cite{Bourn-Gray}, the authors define \defn{algebraically cartesian closed} categories as those for which any change-of-base functor along the terminal map has a right adjoint. It is proven there that, in our context, algebraic cartesian closedness is equivalent to the existence of \emph{centralisers}.

\begin{definition}\label{Def Centraliser}
	Let \(X\oleq H\) be a subobject of an object \(H\). The \textbf{centraliser} \(Z_H(X)\) of \(X\) in \(H\) is the largest subobject of \(H\) which cooperates with \(X\).
\end{definition}

\begin{example}
	In \(\HSLat\), the centraliser of \(X\) in \(H\) is given by
	\[
		Z_H(X)
		= \{h\in H \mid \text{for each \(x\in X\), \((x\iimplies h)=h\) and \((h\iimplies x)=x\)}\}\, .
	\]
\end{example}

This observation, made in~\cite{HSLat}, allowed us to prove the following.

\begin{proposition}[\cite{HSLat}]\label{Proposition Centralisers}
	The category \(\HSLat\) is algebraically cartesian closed. Moreover, centralisers of normal subobjects are normal in~\(\HSLat\). \noproof
\end{proposition}

In general, it is not true in \(\HSLat\) that the centraliser of an arbitrary subobject is normal. For instance, in Example~\ref{example non semi-boolean}, the centraliser of \( \{x,1\}\) is \( \{y,1\}\), which is not up-closed. However, the situation improves in \(\ELHSLat\), as shown below.

\begin{proposition}\label{proposition centraliser}
	Let \(H\) be an object of \(\ELHSLat\), and let \(X \oleq H\) be a subobject. Then the centraliser \(Z_H(X)\) of \(X\) in \(H\) is a normal subobject.
\end{proposition}

\begin{proof}
	To show that \(Z_H(X)\) is normal, it suffices to prove that it is a filter in \(H\). Let \(h'\in H\) be such that \(h' \geq h\) for some \(h\in Z_H(X)\). Recall (see Proposition~\ref{proposition commute in hslat}) that \(h\in Z_H(X)\) is equivalent to the condition that
	\[
		h \vee x = 1 \quad \text{for all \(x\in X\).}
	\]
	Now, applying the third item of Proposition~\ref{proposition equivalences semi-boolean}, we have
	\[
		1 = h \vee x \;\leq\; h' \vee x .
	\]
	Thus \(h'\vee x = 1\) for all \(x\in X\), which implies \(h' \in Z_H(X)\). Hence \(Z_H(X)\) is up-closed, and therefore normal.
\end{proof}

Using Proposition~\ref{proposition commutator is intersection}, one easily sees that in \(\HSLat\), if the intersection of the normal closures \(\bar X\) and \(\bar Y\) of two subobjects \(X\), \(Y \oleq H\) is trivial, then \(X\) and~\(Y\) cooperate in \(H\). This observation is in fact valid in any semi-abelian category (as a direct consequence of Corollary 2.17 from~\cite{HSLat}). The converse does not hold in general, as can again be seen from Example~\ref{example non semi-boolean}. However, in the case of equationally linear Heyting semilattices, these conditions are equivalent.

\begin{proposition}\label{proposition elhslat huq}
	Let \(H\) be an object of \(\ELHSLat\), and let \(X\), \(Y \oleq H\) be two subobjects. Then, the subobjects \(X\) and \(Y\) Huq-commute in \(H\) if and only if their normal closures \(\bar X\) and \(\bar Y\) have trivial intersection.
\end{proposition}
\begin{proof}
	Suppose that \(X\) and \(Y\) Huq-commute. Then, by definition of a centraliser, we have that \(Y \oleq Z_H(X)\). Since \(Z_H(X)\) is normal in \(H\) by Proposition~\ref{proposition centraliser}, it follows that \(\bar Y \oleq Z_H(X)\), and hence \(X\) and \(\bar Y\) Huq-commute.
	Applying the same argument again, we obtain \(\bar X \oleq Z_H(\bar Y)\), so that \(\bar X\) and \(\bar Y\) Huq-commute.

	We conclude using the fact that two subobjects Huq-commute if and only if their Higgins commutator vanishes, together with Proposition~\ref{proposition commutator is intersection}.
\end{proof}

For the sake of completeness, we remark that a more element-wise (and less categorical) proof of the previous fact can be obtained using the following lemma.

\begin{lemma}
	Let \(S\) be a non-empty subset of an equationally linear Heyting semilattice \(H\), and suppose that \(S\) is closed under meets.
	Then the normal subobject \( \bar S \) generated by \(S\) in \(H\) is
	\[
		\bar S \;=\; \{\, s \vee h \mid s \in S,\ h \in H \,\}.
	\]
\end{lemma}
\begin{proof}
	We first prove that \(\bar S\) is a filter of \(H\).
	Let \(h' \in H\) be such that \(h' \ge s \vee h\) for some \(s\in S\) and \(h\in H\).
	Then \(h' = s \vee (h \vee h')\), so \(h'\in \bar S\).
	To show that \(\bar S\) is closed under meets, take \(s,s'\in S\) and \(h,h'\in H\); then
	\begin{align*}
		(s\vee h) \wedge (s'\vee h')
		 & = (s\wedge(s'\vee h')) \vee (h \wedge (s'\vee h'))                               \\
		 & = (s\wedge s') \vee (s\wedge h') \vee (h\wedge s') \vee (h\wedge h') \in \bar S,
	\end{align*}
	where we used that \(S\) is closed under meets and the distribution of the meet over the join given by Proposition~\ref{proposition equivalences semi-boolean}.

	Now let us show that \(\bar S\) is the smallest normal subobject of \(H\) containing \(S\).
	Let \(K \normal H\) with \(S \oleq K\).
	For any \(s\in S\) and \(h\in H\), we have \(s \le s\vee h\); since \(K\) is a filter and contains \(s\), it must also contain \(s\vee h\).
	Thus \(\bar S \oleq K\), which concludes the proof.
\end{proof}

Now, in order to obtain an alternative proof of Proposition~\ref{proposition elhslat huq}, suppose that \(z\) belongs to \(\bar X \omeet \bar Y\).
Then
\[
	z = x \vee h = y \vee h' = z \vee z = x \vee h \vee y \vee h',
\]
for some \(x\in X\), \(y\in Y\) and \(h\), \(h'\in H\).
If \(X\) and \(Y\) cooperate, then \(x \vee y = 1\) by Proposition~\ref{proposition commute in hslat}, and therefore \(z = 1\).

The property stated in Proposition~\ref{proposition elhslat huq} is a remarkably strong one. The only other semi-abelian category we know that satisfies it is the category of Boolean rings without unit. The reason is that this property is, as we establish below in Theorem~\ref{theorem arithmetical and w}, equivalent to \emph{arithmeticity together with the condition~\W{}}.

The condition \W{}, introduced in~\cite{MFVdL3}, essentially requires that \emph{weighted commutators}~\cite{GJU} be independent of the chosen weight. It is a strong restriction: the category \(\Grp\) of groups does not satisfy it (as shown in~\cite{MFVdL3}), and neither does the category \(\HSLat\) of Heyting semilattices, as demonstrated in~\cite{HSLat} using the following equivalent characterisation of~\W{}, established there.

\begin{theorem}\label{theorem w}
	Let \(\C\) be a semi-abelian category. The following conditions are equivalent:
	\begin{tfae}
		\item \(\C\) satisfies~\SH{} (see~\cite{MFVdL}), and for all subobjects \(X\), \(Y \oleq A\),
		\[
			\text{\([X,Y] = 0\) \quad if and only if \quad \([\bar X, \bar Y] = 0\);}
		\]
		\item \(\C\) satisfies~\W{}.\noproof
	\end{tfae}
\end{theorem}

\begin{theorem}\label{theorem arithmetical and w}
	Let \(\C\) be a semi-abelian category. The following conditions are equivalent:
	\begin{tfae}
		\item for all subobjects \(X\), \(Y \oleq A\),
		\[
			[X,Y] = 0 \quad \text{if and only if} \quad \bar X \omeet \bar Y = 0;
		\]
		\item \(\C\) is arithmetical and satisfies condition~\W{}.
	\end{tfae}
\end{theorem}

\begin{proof}
	First observe that, taking \(X = Y = A\), one sees that \(A\) cooperates with itself if and only if \(A\) is the zero object. By Theorem~3.2 of~\cite{HSLat}, this implies that \(\C\) is arithmetical. Since any arithmetical category satisfies~\SH{}~\cite{MFVdL3}, Theorem~\ref{theorem w} reduces us to showing that for all subobjects \(X\), \(Y \oleq A\),
	\[
		\text{\([X,Y] = 0\) \quad if and only if \(\quad [\bar X, \bar Y] = 0\).}
	\]
	But when \(\C\) is arithmetical, we have \([\bar X,\bar Y] = \bar X \omeet \bar Y\) by Theorem~3.7 of~\cite{HSLat}, which yields the claim.

	The converse direction follows similarly from Theorem~3.7 of~\cite{HSLat}.
\end{proof}

\begin{corollary}
	The category \(\ELHSLat\) satisfies the condition \W{}.\noproof
\end{corollary}

In~\cite{HSLat}, we were not able to give a precise description of commutators in Heyting semilattices.
However, the situation simplifies in equationally linear Heyting semilattices.

\begin{proposition}
	Let \(H\) be an object of \(\ELHSLat\), and let \(X\), \(Y \oleq H\) be two subobjects. The Huq-commutator of \(X\) and \(Y\) is given by
	\[
		[X,Y]^{\Huq} = \bigl\lbrace (x \vee y) \vee h \;\big|\; x \in X,\, y \in Y,\, h \in H \bigr\rbrace.
	\]
\end{proposition}

\begin{proof}
	By definition, together with Proposition~\ref{proposition commute in hslat}, the Huq-commutator must be the smallest normal subobject of \(H\) containing all elements of the form \(x \vee y\) with \(x \in X\) and \(y \in Y\).

	First, observe that it is indeed a filter. It is clearly up-closed: if \(h' \in H\) satisfies \(h' \geq (x \vee y) \vee h\) for some \(x \in X\), \(y \in Y\) and \(h \in H\), then
	\[
		h' = (x \vee y) \vee (h \vee h'),
	\]
	so \(h' \in [X,Y]^{\Huq}\).

	To see that it is closed under meets, take any \(x\), \(x' \in X\), \(y\), \(y' \in Y\) and \(h\), \(h' \in H\). Using the distributivity from Proposition~\ref{proposition equivalences semi-boolean}, we compute:
	\begin{align*}
		(x \vee y \vee h) \wedge (x' \vee y' \vee h') =
		 & \;(x \wedge x') \vee (x \wedge y') \vee (x \wedge h') \vee (y \wedge x') \vee (y \wedge y') \\ & \vee (y \wedge h')
		\vee (h \wedge x') \vee (h \wedge y') \vee (h \wedge h').
	\end{align*}
	This element is clearly greater or equal to \((x \wedge x') \vee (y \wedge y')\), which itself belongs to \([X,Y]^{\Huq}\). Since \([X,Y]^{\Huq}\) is up-closed, the whole meet lies in \([X,Y]^{\Huq}\).

	Finally, suppose that \(K\) is a normal subobject of \(H\) containing every \(x \vee y\) with \(x \in X\) and \(y \in Y\). Then for any \(h \in H\),
	\[
		(x \vee y) \vee h \ge x \vee y,
	\]
	and since \(K\) is up-closed, we have \((x \vee y) \vee h \in K\). Hence \(K\) contains \([X,Y]^{\Huq}\).
\end{proof}

\section{Algebraic coherence}\label{section ac}

In the context of a semi-abelian category, the condition called \emph{algebraic coherence} \AC{} introduced in~\cite{acc} has strong categorical-algebraic consequences. Nevertheless, many semi-abelian categories do satisfy it: all varieties of groups, Lie algebras, Poisson algebras, rings without unit, associative algebras, the category of cocommutative Hopf algebras over a field of characteristic~0, all abelian categories, etc. Non-examples include loops and digroups, non-associative rings, \(n\)-Lie algebras, and Jordan algebras.

In~\cite{HSLat}, we proved that \(\HSLat\) is not an algebraically coherent category. The goal of this section is to show that this changes for the category \(\ELHSLat\).

Recall from~\cite{acc} that a category with finite limits is called \defn{algebraically coherent} if and only if, for each morphism \(p\colon E \to B\), the change-of-base functor \(p^*\) is coherent, meaning that it preserves finite limits and jointly strongly epimorphic pairs. It was shown there that regular functors between regular categories are coherent if and only if they preserve binary joins. Moreover, in the case of pointed protomodular categories, algebraic coherence is equivalent to requiring all kernel functors \(\Ker_B\) to be coherent. Therefore, we obtain the following theorem, which is closely related to Theorem~3.15 in~\cite{AlgebraicLogoi}:

\begin{theorem}\label{theorem equivalences ac}
	Let \(\C\) be a semi-abelian category. The following are equivalent:
	\begin{tfae}
		\item \(\C\) is algebraically coherent;
		\item for any object \(B\in \C\), the kernel functor \(\Ker_B \colon \Pt_B(\C) \rightarrow \C\) preserves binary joins;
		\item for any cospan of sub-split-extensions of the form
		\begin{center}
			\begin{tikzcd}[column sep=large,row sep=large]
				X_1 \arrow[r,"\kappa_1"] \arrow[d,"u_1",',tail]
				& A_1 \arrow[r,shift left,"\alpha_1"] \arrow[d,"v_1",',tail]
				& B \arrow[l,shift left,"\beta_1"] \arrow[d,equal] \\
				X \arrow[r,"\kappa"]
				& A \arrow[r,shift left,"\alpha"]
				& B \arrow[l,shift left,"\beta"] \\
				X_2 \arrow[r,"\kappa_2"'] \arrow[u,"u_2",tail]
				& A_2 \arrow[r,shift left,"\alpha_2"] \arrow[u,"v_2",tail]
				& B \arrow[l,shift left,"\beta_2"] \arrow[u,equal]
			\end{tikzcd}
		\end{center}
		with \(A=A_1\ojoin A_2\), we have that \((X\omeet A_1)\ojoin (X\omeet A_2) =X\).\noproof
	\end{tfae}

\end{theorem}

Our strategy in this section is as follows. We mainly rely on the third item of
Theorem~\ref{theorem equivalences ac} and first show that it holds in \(\ELHSLat\)
in the special case where the object \(B\) is the two-element set \(\{b,1\}\).
We then extend the argument to arbitrary finite equationally linear Heyting
semilattices \(B\), and finally we treat the general case.

In what follows, we write \(h \lhd h' \coloneq (h' \iimplies h)\iimplies h'\).
The next lemma provides identities that will be used throughout this section;
its proof is postponed to Appendix~\ref{Appendix}.

\begin{lemma}\label{Equations}
	The following identities hold for any equationally linear Heyting semilattice:
	\begin{enumerate}
		\item \(b\lhd(x\iimplies y) = (b\lhd x) \iimplies (b\lhd y)\),
		\item \(b\lhd(x\wedge y) = (b\lhd x) \wedge (b \lhd y)\),
		\item \(b\iimplies x = (b\vee x) \iimplies (b\lhd x)\),
		\item \((x\iimplies b) \iimplies b = ((b \lhd  x) \iimplies (b \vee x))\iimplies (b \vee x)\),
		\item \(b\vee (x\iimplies y) = (((b\lhd x) \iimplies (b\lhd y)) \iimplies (b \vee x)) \iimplies (((b\lhd y) \iimplies (b\lhd x))\iimplies (b \vee y))\),
		\item \((x\iimplies b) \iimplies y = (((b \lhd  x) \iimplies (b \vee y))\iimplies (b \vee y)) \wedge ((b \vee y) \iimplies (b\lhd y))\).
	\end{enumerate}
\end{lemma}

\subsection*{Algebraic coherence over the two-element set}

Our main aim now is to show that, in \(\ELHSLat\), the kernel functor \(\Ker_B\) is coherent whenever \(B=\{b,1\}\).
We consider a cospan of sub--split extensions as in Theorem~\ref{theorem equivalences ac}, with \(B=\{b,1\}\), and we write
\(X_i \coloneq X \omeet A_i\) for \(i=1,2\).
To prove Proposition~\ref{theorem: algebraic coherence for 2}, we proceed by establishing a series of intermediate lemmas.

\begin{lemma}\label{lemma: useful terms}
	For each \(x \in X_1 \ojoin X_2\), the elements \(b \lhd x\) and \(x \vee b\) belong to \(X_1 \ojoin X_2\).
\end{lemma}
\begin{proof}
	Let
	\[
		\tilde X \coloneq \{x \in X_1 \ojoin X_2 \mid \text{\(b \lhd x\) and \(b \vee x\) belong to \(X_1 \ojoin X_2\)}\}.
	\]
	Noting that each \(X_i\) is contained in \(\tilde X\), it is sufficient to show that \(\tilde X\) is a subalgebra of \(X_1 \ojoin X_2\) since this would imply that \(\tilde X=X_1\ojoin X_2\).
	Suppose that \(x\), \(y \in \tilde X\). By Lemma~\ref{Equations}, we have
	\begin{align*}
		b\lhd (x\iimplies y) & = (b\lhd x) \iimplies (b\lhd y),                         \\
		b\lhd (x\wedge y)    & = (b\lhd x) \wedge (b\lhd y),                            \\
		b\vee (x\iimplies y) & = (((b\lhd x) \iimplies (b\lhd y)) \iimplies (b \vee x))
		\iimplies (((b\lhd y) \iimplies (b\lhd x)) \iimplies (b \vee y)),               \\
		b\vee (x\wedge y)    & = (b \vee x) \wedge (b \vee y).
	\end{align*}
	Hence \(\tilde X\) is closed under \(\wedge\) and \(\iimplies\). Since \(1 \in \tilde X\) trivially, this concludes the proof.
\end{proof}

\begin{lemma}\label{lemma: (x-> b)->b and b->x}
	For each \(x \in X_1 \ojoin X_2\), the elements \((x\iimplies b)\iimplies b\) and \(b\iimplies x\) belong to \(X_1 \ojoin X_2\).
\end{lemma}
\begin{proof}
	By equations~(3) and~(4) of Lemma~\ref{Equations}, we have
	\begin{align*}
		(x\iimplies b) \iimplies b & = ((b \lhd x) \iimplies (b \vee x)) \iimplies (b \vee x), \\
		b\iimplies x               & = (b \vee x) \iimplies (b \lhd x).
	\end{align*}
	The conclusion then follows from Lemma~\ref{lemma: useful terms}.
\end{proof}

\begin{lemma}\label{lemma: (x->b)->y}
	For each \(x\), \(y \in X_1 \ojoin X_2\), the element \((x\iimplies b)\iimplies y\) belongs to \(X_1 \ojoin X_2\).
\end{lemma}
\begin{proof}
	By equation~(6) of Lemma~\ref{Equations}, we have
	\begin{align*}
		(x\iimplies b) \iimplies y
		 & = \bigl((b \lhd x) \iimplies (b \vee y)\bigr)\iimplies (b \vee y)
		\;\wedge\; \bigl((b \vee y) \iimplies (b \lhd y)\bigr).
	\end{align*}
	The conclusion then follows from Lemma~\ref{lemma: useful terms}.
\end{proof}

We denote by \(\tilde A\) the subset
\[
	\tilde A \coloneq
	\{\, a \in A \mid \text{\(a \in X_1 \ojoin X_2\) or \(a = x \wedge (y \iimplies b)\) for some \(x\), \(y \in X_1 \ojoin X_2\)}\}\text{.}
\]
Our aim is to show that \(\tilde A = A\), which will immediately imply
Proposition~\ref{theorem: algebraic coherence for 2}, as explained in its proof.

\begin{lemma}\label{lemma: meet}
	\(\tilde A\) is closed under finite meets.
\end{lemma}
\begin{proof}
	Suppose that \(a\), \(a' \in \tilde A\). The only case that requires attention is when
	\[
		a = x \wedge (y \iimplies b)
		\quad\text{and}\quad
		a' = x' \wedge (y' \iimplies b).
	\]
	In this case we have
	\[
		a \wedge a' = (x \wedge x') \wedge \bigl((y \vee y')\iimplies b\bigr),
	\]
	which belongs to \(\tilde A\) since \(X_1 \ojoin X_2\) is closed under \(\wedge\) and \(\vee\).
\end{proof}

\begin{remark}
	This lemma highlights the importance of \(B\) being the two-element set.
	If \(a' = x' \wedge (y' \iimplies b')\) for some other \(b' \in B\),
	we could not rewrite \(a \wedge a'\) in the desired form.
\end{remark}

\begin{lemma}\label{lemma x -> a and a->x}
	For \(a \in \tilde A\) and \(x \in X_1 \ojoin X_2\), both \(x \iimplies a\) and \(a \iimplies x\) belong to \(\tilde A\).
\end{lemma}
\begin{proof}
	If \(a \in X_1 \ojoin X_2\), then \(x \iimplies a\) and \(a \iimplies x\) clearly belong to \(X_1 \ojoin X_2\).

	Now, suppose that \(a = x' \wedge (x'' \iimplies b)\) for some \(x',x'' \in X_1 \ojoin X_2\). Then
	\[
		a \iimplies x = x' \iimplies \bigl((x'' \iimplies b)\iimplies x\bigr),
	\]
	which lies in \(X_1 \ojoin X_2\) by Lemma~\ref{lemma: (x->b)->y}. On the other hand,
	\[
		x \iimplies a = (x \iimplies x') \wedge \bigl((x \wedge x'')\iimplies b\bigr),
	\]
	which belongs to \(\tilde A\) by definition.
\end{proof}

We note that the identity \((x \wedge y) \iimplies z = x \iimplies (y \iimplies z)\) was used twice in the previous proof; it is easily verified in any Heyting semilattice. We use it again in order to prove the two following lemmas.

\begin{lemma}\label{lemma: a-->b}
	For any \(a \in \tilde A\), the element \(a \iimplies b\) belongs to \(\tilde A\).
\end{lemma}
\begin{proof}
	If \(a \in X_1 \ojoin X_2\), then
	\[
		a \iimplies b = 1 \wedge (a \iimplies b),
	\]
	which belongs to \(\tilde A\).
	Suppose now that \(a = x \wedge (y \iimplies b)\) for some \(x\), \(y \in X_1 \ojoin X_2\).
	Then
	\[
		a \iimplies b = x \iimplies \bigl((y \iimplies b)\iimplies b\bigr),
	\]
	which lies in \(X_1 \ojoin X_2\) by Lemma~\ref{lemma: (x-> b)->b and b->x}.
\end{proof}

\begin{lemma}\label{lemma: implication}
	\(\tilde A\) is closed under implications.
\end{lemma}
\begin{proof}
	Suppose that \(a\), \(a' \in \tilde A\). If \(a' \in X_1 \ojoin X_2\), then \(a \iimplies a'\) belongs to \(\tilde A\) by Lemma~\ref{lemma x -> a and a->x}. If \(a' = x \wedge (y \iimplies b)\) for some \(x\), \(y\in X_1\ojoin X_2\), then
	\[
		a \iimplies a' = (a \iimplies x) \wedge \bigl(y \iimplies (a \iimplies b)\bigr).
	\]
	By Lemma~\ref{lemma: a-->b}, the element \(a \iimplies b\) belongs to \(\tilde A\).
	Lemma~\ref{lemma x -> a and a->x} then implies that both
	\(a \iimplies x\) and \(y \iimplies (a \iimplies b)\) lie in \(\tilde A\).
	Finally, Lemma~\ref{lemma: meet} shows that \(a \iimplies a'\) belongs to~\(\tilde A\).
\end{proof}

\begin{proposition}\label{theorem: algebraic coherence for 2}
	In the category \(\ELHSLat\), for any cospan of sub-split-extensions of the form
	\begin{center}
		\begin{tikzcd}[column sep=large,row sep=large]
			X_1 \arrow[r,"\kappa_1"] \arrow[d,"u_1",',tail]
			& A_1 \arrow[r,shift left,"\alpha_1"] \arrow[d,"v_1",',tail]
			& B \arrow[l,shift left,"\beta_1"] \arrow[d,equal] \\
			X \arrow[r,"\kappa"]
			& A \arrow[r,shift left,"\alpha"]
			& B \arrow[l,shift left,"\beta"] \\
			X_2 \arrow[r,"\kappa_2"'] \arrow[u,"u_2",tail]
			& A_2 \arrow[r,shift left,"\alpha_2"] \arrow[u,"v_2",tail]
			& B \arrow[l,shift left,"\beta_2"] \arrow[u,equal]
		\end{tikzcd}
	\end{center}
	with \(A=A_1\ojoin A_2\) and \(B=\{ b,1\}\), we have that \((X\omeet A_1)\ojoin (X\omeet A_2) =X\).
\end{proposition}
\begin{proof}
	It remains to prove that \(\tilde A = A\). Indeed, since \(X\) is contained in \(A\), this would imply that for each \(x \in X\) we have either \(x \in X_1 \ojoin X_2\) or
	\[
		x = x' \wedge (y' \iimplies b)
	\]
	for some \(x',y' \in X_1 \ojoin X_2\). However, the latter is impossible, since
	\[
		\alpha\bigl(x' \wedge (y' \iimplies b)\bigr) = b,
	\]
	and thus \(x' \wedge (y' \iimplies b)\) cannot lie in the kernel of \(\alpha\). Therefore, we conclude that \(X_1 \ojoin X_2 = X\).

	Now, applying Remark~\ref{Remark Split Extension}, we know that every element \(a \in A_i\) can be written as
	\[
		a = x \wedge (y \iimplies b')
	\]
	for some \(x\), \(y \in X_i\) and \(b' \in B\). This shows that both \(A_1\) and \(A_2\) are contained in \(\tilde A\). Moreover, \(\tilde A\) trivially contains \(1\), and it follows from Lemma~\ref{lemma: meet} and Lemma~\ref{lemma: implication} that \(\tilde A\) is a subalgebra of \(A\). Hence \(\tilde A = A\), which concludes the proof.
\end{proof}

\subsection*{\texorpdfstring{\(\ELHSLat\) is algebraically coherent}{ELHSLat is algebraically coherent}}

We now turn to the proof of algebraic coherence for equationally linear Heyting semilattices.
We begin by establishing the coherence of the kernel functor \(\Ker_B\) in the case where \(B\) is finite, and we then extend the result to the general case.
Before doing so, we start with the following lemma.

\begin{lemma}\label{lemma on coherence}
	In a pointed category \(\C\) with finite limits, for any split epimorphism
	\(p\colon E \to B\) with kernel \(k\colon N \to E\), if the kernel functor
	\[
		\Ker_B \colon \Pt_B(\C) \to \C
	\]
	preserves binary joins, then the change-of-base functor
	\[
		k^* \colon \Pt_E(\C) \to \Pt_N(\C)
	\]
	preserves binary joins as well.
\end{lemma}

\begin{proof}
	Let \(s\colon B \to E\) be a splitting of \(p\).
	Since pullbacks preserve kernels in pointed categories with finite limits, the diagram
	\[
		\xymatrix{
			\Pt_E(\C) \ar[r]^{k^*} \ar[d]_{Q} &
			\Pt_N(\C) \ar[d]^{P} \\
			\Pt_B(\C) \ar[r]_{\Ker_B} &
			\C
		}
	\]
	commutes up to isomorphism. Here, \(P\) is the functor sending a point over \(N\) to the domain of its split epimorphism, and \(Q\) sends a point \((C,\gamma,\delta)\) to \((C,p\gamma,\delta s)\).

	Now suppose that \(\Ker_B\) preserves finite joins. Since finite joins in the categories of points are computed as in \(\C\), we obtain
	\begin{align*}
		P k^*(C_1 \ojoin C_2)
		 & \cong \Ker_B Q(C_1 \ojoin C_2)               \\
		 & \cong \Ker_B Q(C_1) \ojoin \Ker_B Q(C_2)     \\
		 & \cong P k^*(C_1) \ojoin P k^*(C_2)           \\
		 & \cong P\bigl(k^*(C_1) \ojoin k^*(C_2)\bigr).
	\end{align*}
	We conclude that
	\[
		k^*(C_1 \ojoin C_2) \cong k^*(C_1) \ojoin k^*(C_2),
	\]
	as required.
\end{proof}

\begin{proposition}\label{proposition ac for finite}
	For each finite equationally linear Heyting semilattice \(B\), the kernel functor
	\[
		\Ker_B \colon \Pt_B(\ELHSLat) \longrightarrow \ELHSLat
	\]
	preserves binary joins.
\end{proposition}

\begin{proof}
	Let \(B\) be a finite equationally linear Heyting semilattice with \(n\) elements.
	The statement is clear for \(n=1\), and the case \(n=2\) was established in
	Proposition~\ref{theorem: algebraic coherence for 2}.
	Assume that the claim holds for any finite equationally linear Heyting semilattice
	with fewer than \(n\) elements.

	Since \(B\) is finite, it is a Heyting algebra, and there exists a unique Heyting
	algebra morphism \(s\colon 2 \to B\), where \(2 \coloneq \{0,1\}\) is the initial object
	in the category of Heyting algebras.
	Viewed as a morphism of Heyting semilattices, \(s\) is a split monomorphism, with
	splitting \(p\colon B \to 2\) given by quotienting by the up-set
	\[
		\ua{x} \coloneq \{ y \in B \mid y \ge x \}
	\]
	of any minimal element \(x \in B\) amongst elements above the bottom.
	We denote by \(k\colon N \to B\) the kernel of \(p\).

	Since \(N\) has strictly fewer elements than \(B\), the functor \(\Ker_N\) preserves
	binary joins by the induction hypothesis.
	Applying Lemma~\ref{lemma on coherence}, together with the fact that \(\Ker_2\)
	preserves binary joins by Proposition~\ref{theorem: algebraic coherence for 2}, we
	conclude that the change-of-base functor \(k^*\) preserves binary joins.
	Finally, we observe that
	\[
		\Ker_B \cong \Ker_N  k^*,
	\]
	which holds in any pointed category with finite limits. This completes the proof.
\end{proof}

\begin{theorem}\label{theorem AC for elhslat}
	The category of equationally linear Heyting semilattices is algebraically coherent.
\end{theorem}
\begin{proof}
	Consider a cospan of sub-split extensions of the form
	\begin{center}
		\begin{tikzcd}[column sep=large,row sep=large]
			X_1 \arrow[r,"\kappa_1"] \arrow[d,"u_1",',tail]
			& A_1 \arrow[r,shift left,"\alpha_1"] \arrow[d,"v_1",',tail]
			& B \arrow[l,shift left,"\beta_1"] \arrow[d,equal] \\
			X \arrow[r,"\kappa"]
			& A \arrow[r,shift left,"\alpha"]
			& B \arrow[l,shift left,"\beta"] \\
			X_2 \arrow[r,"\kappa_2"'] \arrow[u,"u_2",tail]
			& A_2 \arrow[r,shift left,"\alpha_2"] \arrow[u,"v_2",tail]
			& B \arrow[l,shift left,"\beta_2"] \arrow[u,equal]
		\end{tikzcd}
	\end{center}
	with \(A = A_1 \ojoin A_2\).
	In order to prove algebraic coherence, by Theorem~\ref{theorem equivalences ac}, it suffices to show that
	\[
		X = X_1 \ojoin X_2,
	\]
	where \(X_1 = X \omeet A_1\) and \(X_2 = X \omeet A_2\).

	Let \(T\) be the set of all finite subalgebras of \(B\). Clearly,
	\[
		B = \bigcup_{S \in T} S.
	\]
	Recall that in Heyting semilattices every finitely generated algebra is finite (see~\cite{NeWh} or Appendix~\ref{Appendix}).

	By algebraic coherence over finite objects (Proposition~\ref{proposition ac for finite}), for any \(S \in T\) we have
	\begin{align*}
		 & X \omeet \bigl((A_1 \omeet \alpha^{-1}(S)) \ojoin (A_2 \omeet \alpha^{-1}(S))\bigr)    \\
		 & = (X \omeet (A_1 \omeet \alpha^{-1}(S))) \ojoin (X \omeet (A_2 \omeet \alpha^{-1}(S))) \\
		 & = X_1 \ojoin X_2 .
	\end{align*}
	Here we used the fact that \(\alpha^{-1}(S)\) contains the kernel of \(\alpha\).
	Since \(\alpha\) is a split epimorphism, it follows that
	\[
		A = \bigcup \{ \alpha^{-1}(S) \mid S \in T \}.
	\]
	Therefore,
	\begin{align*}
		A
		 & = (A_1 \omeet A) \ojoin (A_2 \omeet A)                                                                   \\
		 & = \bigojoin_{S \in T} (A_1 \omeet \alpha^{-1}(S)) \ojoin \bigojoin_{S \in T} (A_2 \omeet \alpha^{-1}(S)) \\
		 & = \bigojoin_{S \in T} \bigl((A_1 \omeet \alpha^{-1}(S)) \ojoin (A_2 \omeet \alpha^{-1}(S))\bigr).
	\end{align*}
	Intersecting with \(X\), we obtain
	\begin{align*}
		X
		 & = \bigcup \{ X \omeet ((A_1 \omeet \alpha^{-1}(S)) \ojoin (A_2 \omeet \alpha^{-1}(S))) \mid S \in T \} \\
		 & = \bigcup \{ X_1 \ojoin X_2 \mid S \in T \}                                                            \\
		 & = X_1 \ojoin X_2,
	\end{align*}
	as required.
\end{proof}

In~\cite{acc}, Section~8, the authors provide a summary of conditions that follow from algebraic coherence.
Applying this list in our setting, we obtain the following.

\begin{corollary}
	The category \(\ELHSLat\) satisfies the following properties:
	\begin{enumerate}
		\item change-of-base functors of the fibration of points preserve Higgins and Huq commutators, normal closures, and cokernels;
		\item it satisfies the conditions \SH{} and \NH{};
		\item it is peri-abelian and therefore satisfies the universal central extension condition;
		\item it satisfies \SSH{};
		\item it is strongly protomodular;
		\item it is fibrewise algebraically cartesian closed \FWACC{};
		\item for any normal subobjects \(K,L,M\) of an object \(X\), we have
		      \[
			      [K,L,M] = [[K,L],M] \ojoin [L,[K,M]]
		      \]
		      as subobjects of \(X\).\noproof
	\end{enumerate}
\end{corollary}

Combining this with Proposition~\ref{proposition elhslat huq} and Theorem~\ref{theorem w}, we obtain the following.

\begin{corollary}
	Every fibre of the fibration of points in \(\ELHSLat\) satisfies the condition \W{}.
\end{corollary}

\begin{proof}
	Let \(S\) and \(T\) be subobjects of an object \(A\) in \(\Pt_B(\ELHSLat)\). As explained in~\cite{MFVdL3}, since \SSH{} holds, \(S\) and \(T\) Huq-commute in \(\Pt_B(\ELHSLat)\) if and only if \(\Ker_B(S)\) and \(\Ker_B(T)\) Huq-commute in \(\ELHSLat\). By Proposition~\ref{proposition elhslat huq}, this is equivalent to \(\Ker_B(\bar S)\) and \(\Ker_B(\bar T)\) Huq-commuting.	Finally, this holds if and only if \(\bar S\) and \(\bar T\) Huq-commute in \(\Pt_B(\ELHSLat)\).
\end{proof}

In~\cite{AlgebraicLogoi}, the authors introduced the notion of \emph{algebraic logoi}. A category with finite limits is called an \defn{algebraic logos} if and only if every change-of-base functor preserves jointly strongly epimorphic (small) families of arrows. In particular, any algebraic logos is an algebraically coherent category. Thanks to Corollary~4.4 of~\cite{AlgebraicLogoi}, the converse holds for homological quasi-varieties of algebras, and therefore:

\begin{corollary}
	The category \(\ELHSLat\) is an algebraic logos.\noproof
\end{corollary}

\section{Failure of local algebraic cartesian closedness}\label{section lacc}

Having shown that \(\ELHSLat\) is algebraically coherent, it is natural to ask whether it satisfies the stronger condition of \emph{local algebraic cartesian closedness}, since (essentially by definition of \AC{}) the condition \LACC{} implies \AC{}~\cite{Bourn-Gray,acc}. We prove that this is not the case: \(\ELHSLat\) is algebraically coherent but not locally algebraically cartesian closed. This is the exact algebraic analogue of the situation for Boolean rings~\cite{Gray2012}, the only other semi-abelian category we know to satisfy the equivalence \([X,Y]=0 \Leftrightarrow \bar X \omeet \bar Y = 0\) of Proposition~\ref{proposition elhslat huq}.

Recall from~\cite{Bourn-Gray} that a finitely complete category \(\C\) is \defn{locally algebraically cartesian closed} (it satisfies condition \LACC{}) when, for every \({p\colon E \to B}\) in \(\C\), the change-of-base functor
\[
	p^* \colon \Pt_B(\C) \longrightarrow \Pt_E(\C)
\]
is a left adjoint, that is, admits a right adjoint.

The kernel functor is an instance of such a change-of-base functor. Indeed, for the unique morphism \(\gnab_B \colon 0 \to B\) we have \(\Pt_0(\C) \simeq \C\), and the pullback of a point \((A,\alpha,\beta)\) along \(\gnab_B\) is precisely the kernel of \(\alpha\); thus
\[
	\Ker_B \cong (\gnab_B)^* \colon \Pt_B(\C) \longrightarrow \C.
\]
This gives the following well-known elementary necessary condition, which is all we shall need.

\begin{proposition}\label{proposition lacc preserves coproducts}
	If a pointed category \(\C\) with finite limits is locally algebraically cartesian closed, then for every object \(B\) the kernel functor \(\Ker_B \colon \Pt_B(\C) \to \C\) preserves binary coproducts.
\end{proposition}
\begin{proof}
	Under \LACC{} the functor \(\Ker_B \cong (\gnab_B)^*\) is a left adjoint, hence preserves all colimits, and in particular binary coproducts.
\end{proof}

Concretely, for two points \((C_1,\gamma_1,\delta_1)\) and \((C_2,\gamma_2,\delta_2)\) over \(B\), with kernels \(X_1\) and \(X_2\), the coproduct injections induce a canonical \defn{comparison morphism}
\begin{equation}\label{eq comparison}
	c \colon X_1 + X_2 \longrightarrow \Ker_B(C_1 +_B C_2),
\end{equation}
where \(+\) on the left denotes the coproduct in \(\C\) and \(+_B\) the coproduct in the fibre \(\Pt_B(\C)\). The functor \(\Ker_B\) preserves the coproduct \(C_1 +_B C_2\) if and only if~\eqref{eq comparison} is an isomorphism. By contrast, algebraic coherence is equivalent to the requirement that every such \(c\) be a regular epimorphism; this is the kernel-functor reformulation of \AC{} from~\cite{acc} that underlies Theorem~\ref{theorem equivalences ac}, used above to prove that \(\ELHSLat\) is algebraically coherent. Thus, to separate the two conditions, it suffices to exhibit a single comparison morphism~\eqref{eq comparison} that is a regular epimorphism but not a monomorphism.

We work over the two-element base
\[
	B = \{b < 1\},
\]
and we use two split extensions of \(B\) whose kernels are both isomorphic to the two-element chain \(\{\ast < 1\}\).

\smallskip
\noindent\emph{The first point} is the three-element chain \(A_1 = \{p < x_1 < 1\}\), endowed with
\[
	\alpha_1(p) = b, \quad \alpha_1(x_1) = \alpha_1(1) = 1, \quad \beta_1(b) = p, \quad \beta_1(1) = 1.
\]
Then \(\alpha_1\) is a homomorphism with \(\alpha_1\beta_1 = 1_B\), and its kernel is the two-element chain \(X_1 = \{x_1 < 1\}\); we write \(\kappa_1\colon X_1 \rightarrowtail A_1\) for the inclusion and \(\hat b_1 \coloneq \beta_1(b) = p\). Note the crucial order relation
\begin{equation}\label{eq order in A1}
	\hat b_1 = p \leq x_1 \quad \text{in } A_1.
\end{equation}
This extension is not a product: the ``action'' is nontrivial, as witnessed by
\[
	\hat b_1 \iimplies \kappa_1(x_1) = p \iimplies x_1 = 1 \neq x_1.
\]

\smallskip
\noindent\emph{The second point} is the product (trivial) extension
\[
	A_2 = B \times X_2, \qquad X_2 = \{x_2 < 1\},
\]
with \(\alpha_2 = \mathrm{pr}_B\) the projection, \(\beta_2 = (1_B, 0)\) the canonical inclusion, kernel \(\kappa_2 \colon X_2 \cong \{1\}\times X_2 \rightarrowtail A_2\), and \(\hat b_2 \coloneq \beta_2(b) = (b,1)\). Observe that \(A_2 \cong 2^2\) is the four-element Boolean algebra. The defining feature of a product extension is that its action is trivial:
\begin{equation}\label{eq trivial action}
	\hat b_2 \iimplies \kappa_2(x_2) = (b,1) \iimplies (1,x_2) = (b \iimplies 1, \, 1 \iimplies x_2) = (1, x_2) = \kappa_2(x_2).
\end{equation}

The Hasse diagrams of \(B\), \(A_1\) and \(A_2\) are as follows.
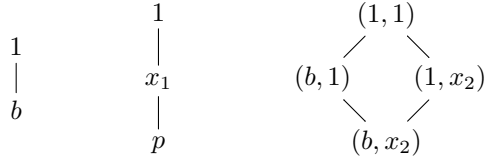
\[
	\vcenter{\xymatrix@!=0ex{
	1 \\ b \ar@{-}[u]
	}}
	\qquad\qquad
	\vcenter{\xymatrix@!=0ex{
	1 \\ x_1 \ar@{-}[u] \\ p \ar@{-}[u]
	}}
	\qquad\qquad
	\vcenter{\xymatrix@!=0ex{
	& (1,1) & \\
	(b,1) \ar@{-}[ur] & & (1,x_2) \ar@{-}[ul] \\
	& (b,x_2) \ar@{-}[ul] \ar@{-}[ur] &
	}}
\]

Form the coproduct \(A_1 +_B A_2\) in the fibre \(\Pt_B(\ELHSLat)\); concretely it is the pushout of \(\beta_1\) and \(\beta_2\) in \(\ELHSLat\), with structural injections \(\iota_1 \colon A_1 \to A_1 +_B A_2\) and \(\iota_2 \colon A_2 \to A_1 +_B A_2\) satisfying \(\iota_1\beta_1 = \iota_2\beta_2\). In particular the two copies of \(b\) are identified:
\[
	\hat b \coloneq \iota_1(\hat b_1) = \iota_2(\hat b_2).
\]
To lighten the notation we write \(x_1\) and \(x_2\) also for their images \(\iota_1\kappa_1(x_1)\) and \(\iota_2\kappa_2(x_2)\) in \(A_1 +_B A_2\); these generate \(\Ker_B(A_1 +_B A_2)\).

\begin{lemma}\label{lemma lacc collapse}
	In \(A_1 +_B A_2\) one has \(x_1 \iimplies x_2 = x_2\).
\end{lemma}
\begin{proof}
	Since \(\iota_1\) is a homomorphism, the order relation~\eqref{eq order in A1} gives \(\hat b \leq x_1\); as the implication is order-reversing in its first variable, this yields
	\[
		x_1 \iimplies x_2 \leq \hat b \iimplies x_2.
	\]
	On the other hand, using \(\iota_1\beta_1 = \iota_2\beta_2\) and the triviality of the action of \(A_2\) from~\eqref{eq trivial action},
	\[
		\hat b \iimplies x_2 = \iota_2(\hat b_2) \iimplies \iota_2\kappa_2(x_2) = \iota_2\bigl(\hat b_2 \iimplies \kappa_2(x_2)\bigr) = \iota_2\kappa_2(x_2) = x_2.
	\]
	Hence \(x_1 \iimplies x_2 \leq x_2\). Since \(y \leq x \iimplies y\) holds in every Heyting semilattice, we also have \(x_2 \leq x_1 \iimplies x_2\), and therefore \(x_1 \iimplies x_2 = x_2\).
\end{proof}

\begin{lemma}\label{lemma lacc free}
	In the coproduct \(X_1 + X_2\) of the kernels, taken in \(\ELHSLat\), one has \(x_1 \iimplies x_2 \neq x_2\).
\end{lemma}
\begin{proof}
	Let \(f \colon X_1 + X_2 \to B = \{b < 1\}\) be the homomorphism determined by the homomorphisms \(X_1 \to B\colon x_1 \mapsto b\) and \(X_2 \to B\colon x_2 \mapsto b\) (each sends the non-top element to \(b\) and is a homomorphism of two-element chains). Then
	\[
		f(x_1 \iimplies x_2) = f(x_1) \iimplies f(x_2) = b \iimplies b = 1, \qquad \text{whereas} \qquad f(x_2) = b.
	\]
	As \(1 \neq b\), the elements \(x_1 \iimplies x_2\) and \(x_2\) of \(X_1 + X_2\) are distinct.
\end{proof}

\begin{theorem}\label{theorem not lacc}
	The variety \(\ELHSLat\) of equationally linear Heyting semilattices is not locally algebraically cartesian closed.
\end{theorem}
\begin{proof}
	Consider the comparison morphism
	\[
		c \colon X_1 + X_2 \longrightarrow \Ker_B(A_1 +_B A_2)
	\]
	of~\eqref{eq comparison} attached to the two points \(A_1\) and \(A_2\) constructed above. By definition \(c\) restricts to \(\iota_i\kappa_i\) on \(X_i\), so it is a homomorphism sending the generators \(x_1\), \(x_2\) of \(X_1 + X_2\) to the elements denoted \(x_1\), \(x_2\) in \(A_1 +_B A_2\). Using Lemma~\ref{lemma lacc collapse},
	\[
		c(x_1 \iimplies x_2) = c(x_1) \iimplies c(x_2) = x_1 \iimplies x_2 = x_2 = c(x_2).
	\]
	By Lemma~\ref{lemma lacc free}, however, \(x_1 \iimplies x_2 \neq x_2\) in \(X_1 + X_2\). Therefore \(c\) is not a monomorphism, so it is not an isomorphism, and consequently \(\Ker_B\) does not preserve the binary coproduct \(A_1 +_B A_2\). By Proposition~\ref{proposition lacc preserves coproducts}, \(\ELHSLat\) is not locally algebraically cartesian closed.
\end{proof}

\section{Failure of cosmash associativity}\label{section cosmash}

We close with another negative result, which settles a question left open in~\cite{HSLat}. Following~\cite{Smash,HVdL}, the \defn{ternary cosmash product} \(X\diamond Y\diamond Z\) of three objects is the kernel of the canonical morphism
\[
	\Sigma_{X,Y,Z}\colon X+Y+Z \longrightarrow (X+Y)\times(X+Z)\times(Y+Z)
\]
whose three components are induced by the morphisms collapsing \(Z\), \(Y\) and \(X\), respectively. Concretely, an element of \(X+Y+Z\) lies in \(X\diamond Y\diamond Z\) precisely when it becomes trivial as soon as any one of the three variables is sent to the zero~\(1\). Iterating the binary cosmash product instead produces \((X\diamond Y)\diamond Z\); there is always a canonical comparison morphism
\[
	c\colon (X\diamond Y)\diamond Z \longrightarrow X\diamond Y\diamond Z,
\]
and, following~\cite{Cosmash}, a category is called \defn{cosmash associative} when \(c\) is an isomorphism for all \(X\), \(Y\), \(Z\). In~\cite{HSLat} it was shown that \(\HSLat\) is not cosmash associative, while it was left open whether the subvariety \(\ELHSLat\) is. We show that it is not.

The comparison \(c\) factors as
\[
	(X\diamond Y)\diamond Z \rightarrowtail (X\diamond Y)+Z \longrightarrow X+Y+Z,
\]
the second arrow being induced by the inclusion \(X\diamond Y\rightarrowtail X+Y\) and the identity on \(Z\). Its image therefore lies in the subalgebra \(\langle X\diamond Y, Z\rangle\) of \(X+Y+Z\) generated by \(X\diamond Y\) and \(Z\). Consequently, if some element of \(X\diamond Y\diamond Z\) lies outside \(\langle X\diamond Y, Z\rangle\), then \(c\) cannot be surjective, and a fortiori cannot be an isomorphism.

\begin{theorem}\label{theorem not cosmash}
	The category \(\ELHSLat\) is not cosmash associative.
\end{theorem}

\begin{proof}
	Let \(X=Y=Z\) be the free equationally linear Heyting semilattice on a single generator, that is, the two-element chain, and write \(x\), \(y\), \(z\) for the respective generators. Then \(X+Y+Z\) is the free equationally linear Heyting semilattice on \(\{x,y,z\}\), which is finite by~\cite{NeWh} (see also Appendix~\ref{Appendix}). Consider the element
	\[
		f \coloneq (x\iimplies y) \iimplies \bigl((z\iimplies x)\iimplies y\bigr).
	\]

	\emph{The element \(f\) belongs to \(X\diamond Y\diamond Z\).} Indeed, sending any single generator to \(1\) reduces \(f\) to \(1\):
	\begin{align*}
		f|_{x=1} & = (1\iimplies y)\iimplies\bigl((z\iimplies 1)\iimplies y\bigr) = y\iimplies(1\iimplies y) = y\iimplies y = 1, \\
		f|_{y=1} & = (x\iimplies 1)\iimplies\bigl((z\iimplies x)\iimplies 1\bigr) = 1\iimplies 1 = 1,                            \\
		f|_{z=1} & = (x\iimplies y)\iimplies\bigl((1\iimplies x)\iimplies y\bigr) = (x\iimplies y)\iimplies(x\iimplies y) = 1,
	\end{align*}
	using the identities \(1\iimplies a = a\) and \(a\iimplies 1 = 1\). Hence \(f\in X\diamond Y\diamond Z\).

	\emph{The element \(f\) does not belong to \(\langle X\diamond Y, Z\rangle\).} We first observe that \(X\diamond Y = \mathord{\uparrow}(x\vee y)\), the principal filter generated by \(x\vee y\) in \(X+Y\). Indeed, since \(X\ojoin Y = X+Y\), the induced morphism \(\langle x,y\rangle\) is the identity, so the Higgins commutator \([X,Y]\) coincides with the cosmash product \(X\diamond Y\); by Proposition~\ref{proposition commutator is intersection} it equals \(X^{X+Y}\omeet Y^{X+Y} = \mathord{\uparrow}x \omeet \mathord{\uparrow}y = \mathord{\uparrow}(x\vee y)\).

	Now consider the homomorphism
	\[
		\varphi\colon X+Y+Z \longrightarrow C\times C', \qquad
		\varphi(x) = (0,0),\quad \varphi(y)=(1,2),\quad \varphi(z)=(0,1),
	\]
	into the product of the chains \(C=\{0<1<2\}\) and \(C'=\{0<1<2<3\}\). A direct computation gives
	\[
		\varphi(f) = (1,3) \qquad\text{and}\qquad \varphi(x\vee y) = (1,2).
	\]
	The image of \(X+Y\) under \(\varphi\) is the subalgebra generated by \((0,0)\) and \((1,2)\), namely the chain \(\{(0,0),(1,2),(2,3)\}\). Hence for every \(w\in X\diamond Y\) we have \(\varphi(w)\in\{(0,0),(1,2),(2,3)\}\) and, since \(w\geq x\vee y\) and \(\varphi\) preserves the order, also \(\varphi(w)\geq(1,2)\); therefore \(\varphi(w)\in\{(1,2),(2,3)\}\). Together with \(\varphi(z)=(0,1)\), this shows that \(\varphi\) maps \(\langle X\diamond Y, Z\rangle\) into the three-element chain
	\[
		D = \{(0,1),\,(1,2),\,(2,3)\},
	\]
	which is a subalgebra of \(C\times C'\) (its top \((2,3)\) being that of the whole product). As \(\varphi(f)=(1,3)\notin D\), we conclude that \(f\notin\langle X\diamond Y, Z\rangle\).

	By the discussion preceding the theorem, the comparison \(c\colon (X\diamond Y)\diamond Z\to X\diamond Y\diamond Z\) is not surjective, hence not an isomorphism. Therefore \(\ELHSLat\) is not cosmash associative.
\end{proof}

The obstruction can be understood as follows: the iterated product \((X\diamond Y)\diamond Z\) can only reach elements of the subalgebra generated by the binary cosmash \(X\diamond Y\) together with \(Z\), which the homomorphism \(\varphi\) collapses onto a single chain. The element \(f\), by contrast, genuinely interleaves the three variables through the nested implication \(z\iimplies x\), and \(\varphi\) sends it to a point off that chain. Since cosmash associativity implies the condition~\W{}~\cite{CocoThesis}, Theorem~\ref{theorem not cosmash} also shows that the converse fails: by the corollary to Theorem~\ref{theorem arithmetical and w}, the variety \(\ELHSLat\) satisfies~\W{} without being cosmash associative.

\appendix
\section{Proof of Lemma~\ref{Equations}}\label{Appendix}

The aim of this appendix is to give a proof of Lemma~\ref{Equations}. To do so, we show that in equationally linear Heyting semilattices, verifying an identity in three variables reduces to checking it in chains of length at most four. Our argument relies on constructing a certain sub-direct embedding of a finitely generated Heyting semilattice into a product of sub-directly irreducible Heyting semilattices and is largely derived from \cite{FreydHeyting}.

Let \(X\) be a Heyting semilattice and let \(x \neq 1\) be a maximal element below the top. Then either \(x\) is the unique maximal element below \(1\), or there exists \(y < 1\) in~\(X\)  such that \(y \nleq x\). Assume we are in the latter case. Choose a maximal normal subobject \(N \oleq X\) such that \(x \notin N\), and let \(q \colon X \to X/N\) denote the quotient map.

We claim that \(q(x)\) is the unique maximal element below the top in \(X/N\). Indeed, let \(y < 1\) in \(X/N\) and consider the inverse image
\[
	N' \coloneq q^{-1}(\ua{y}).
\]
Since \(\ua{y}\) is a normal subobject of \(X/N\), its inverse image \(N'\) is a normal subobject of \(X\) containing \(N\). Moreover, \(N' \supsetneq N\) because \(y < 1\). By the maximality of \(N\), it follows that \(x \in N'\), whence \(q(x) \in \ua{y}\) and therefore \(y \le q(x)\). This shows that~\(q(x)\) is the unique maximal element below \(1\) in \(X/N\). In particular, \((X/N) \setminus \{q(x)\}\) is a sub-Heyting semilattice of \(X/N\).

Following \cite{FreydHeyting}, we call such a quotient a \defn{localisation}. Given a Heyting semilattice~\(A\), we write \(\hat A\) for the Heyting semilattice obtained by adjoining a new element \(m\) satisfying \(a < m < 1\) for all \(a \in A\backslash\{1\}\). The preceding argument shows that every localisation of a Heyting semilattice is of the form \(\hat A\) for some \(A\).

We now explain how this implies that every finitely generated Heyting semilattice is finite. Recall that for elements \(x\), \(y\) of a Heyting semilattice one has \(x = y\) if and only if \((x \leftrightarrow y) = 1\). Hence, if \(x \neq y\), there exists a localisation in which their images remain distinct. It follows that every Heyting semilattice embeds into a product of its localisations.

We proceed now by induction on the number of generators. The claim is trivial for one generator. Assume that every Heyting semilattice generated by less than~\(n\) elements is finite, and let \(X\) be generated by \(n\) elements. Every localisation of \(X\) is of the form \(\hat A\), where \(A\) is generated by \(n-1\) elements, and hence is finite by the induction hypothesis. Consequently, up to isomorphism, \(X\) has only finitely many localisations (since each is of the form \(\hat A\) where \(A\) is a quotient of the free Heyting semilattice on \(n-1\) generators). Since there are only finitely many homomorphisms from a finitely generated algebra to a finite one, it follows that \(X\) embeds into a finite Heyting semilattice, and is therefore itself finite.

Moreover, since any finite Heyting semilattice has maximal elements, one easily sees that it embeds into a product of those localisations that do not send a maximal element to the top. For free Heyting semilattices on a finite set, this allows us to restrict our attention to morphisms into localisations that do not send any generator to the top. Indeed, let \(S\) be a finite set, let \(F(S)\) denote the free Heyting semilattice on \(S\), and let \(f \colon F(S) \to A\) be a surjective homomorphism. If no generator is sent to \(1\), there is nothing to prove. Otherwise, let
\[
	T \coloneq \{x \in S \mid f(x) = 1\},
\]
and define a map \(g \colon S \to \hat A\) by
\[
	g(x) =
	\begin{cases}
		m    & \text{if \(x \in T\),}             \\
		f(x) & \text{if \(x \in S \setminus T\),}
	\end{cases}
\]
where \(m\) denotes the unique maximal element below \(1\) in \(\hat A\). By the universal property of \(F(S)\), \(g\) extends uniquely to a homomorphism \(g \colon F(S) \to \hat A\). Let us define \(h \colon \hat A \to A\) by
\[
	h(x) =
	\begin{cases}
		1 & \text{if \(x = m\),} \\
		x & \text{otherwise}.
	\end{cases}
\]
Then \(h \circ g = f\), proving the claim.

All of the above applies in particular to equationally linear Heyting semilattices. In this case, one can moreover show that all localisations are chains. It follows that finitely generated free equationally linear Heyting semilattices embed into products of finite chains in such a way that no generator is sent to the top. For the free Heyting semilattice on three generators, these chains have length at most four. A direct inspection shows that there is exactly one such homomorphism to the two-element chain, six to the three-element chain (three sending two generators above the third and three sending two below the third), and six to the four-element chain. Consequently, to verify an identity in three variables it suffices to check it in each of these finite chains with generators assigned according to the corresponding homomorphisms.

\begin{proof}[Proof of Lemma~\ref{Equations}]
	Tables~\ref{Table 1}--\ref{Table 4} demonstrate that each of the equations of three variables in Lemma~\ref{Equations} holds in the chain \(\{0,m,n,1\}\) (respectively, \(\{0,m,1\}\) for equations of two variables), and therefore holds throughout \(\ELHSLat\).
	\begin{table}
		\[
			\begin{tabular}{c  c c c c c c c c c c c c c}
				\(b\)                           & \(n\) & \(n\) & \(m\) & \(0\) & \(0\) & \(m\) & \(n\) & \(m\) & \(m\) & \(n\) & \(n\) & \(m\) & \(n\) \\
				\(x\)                           & \(m\) & \(0\) & \(n\) & \(n\) & \(m\) & \(0\) & \(m\) & \(n\) & \(m\) & \(n\) & \(m\) & \(n\) & \(n\) \\
				\(y\)                           & \(0\) & \(m\) & \(0\) & \(m\) & \(n\) & \(n\) & \(m\) & \(m\) & \(n\) & \(m\) & \(n\) & \(n\) & \(n\) \\
				\(b\lhd (x\iimplies y)\)        & \(0\) & \(1\) & \(0\) & \(1\) & \(1\) & \(1\) & \(1\) & \(m\) & \(1\) & \(m\) & \(1\) & \(1\) & \(1\) \\
				\((b\lhd x)\iimplies(b\lhd y)\) & \(0\) & \(1\) & \(0\) & \(1\) & \(1\) & \(1\) & \(1\) & \(m\) & \(1\) & \(m\) & \(1\) & \(1\) & \(1\) \\
				\(b\lhd (x\wedge y)\)           & \(0\) & \(0\) & \(0\) & \(1\) & \(1\) & \(0\) & \(m\) & \(m\) & \(m\) & \(m\) & \(m\) & \(1\) & \(n\) \\
				\((b\lhd x)\wedge (b\lhd y)\)   & \(0\) & \(0\) & \(0\) & \(1\) & \(1\) & \(0\) & \(m\) & \(m\) & \(m\) & \(m\) & \(m\) & \(1\) & \(n\) \\
			\end{tabular}
		\]
		\caption{Identities (1) and (2) of Lemma~\ref{Equations}}\label{Table 1}
	\end{table}
	\begin{table}
		\[
			\begin{tabular}{c c c c}
				\(b\)                                                           &
				\(m\)                                                           & \(0\) & \(m\) \\
				\(x\)                                                           &
				\(0\)                                                           & \(m\) & \(m\) \\
				\(b\iimplies x\)                                                &
				\(0\)                                                           & \(1\) & \(1\) \\
				\((b\vee x)\iimplies (b\lhd x)\)                                &
				\(0\)                                                           & \(1\) & \(1\) \\
				\((x\iimplies b)\iimplies b\)                                   &
				\(m\)                                                           & \(1\) & \(m\) \\
				\(\bigl((b\lhd x)\iimplies (b\vee x)\bigr)\iimplies (b\vee x)\) &
				\(m\)                                                           & \(1\) & \(m\) \\
			\end{tabular}
		\]
		\caption{Identities (3) and (4) of Lemma~\ref{Equations}}\label{Table 2}
	\end{table}
	\begin{table}
		\[
			\begin{tabular}{c c c c c c c c c c c c c c}
				\(b\)                    & \(n\) & \(n\) & \(m\) & \(0\) & \(0\) & \(m\) & \(n\) & \(m\) & \(m\) & \(n\) & \(n\) & \(m\) & \(n\) \\
				\(x\)                    & \(m\) & \(0\) & \(n\) & \(n\) & \(m\) & \(0\) & \(m\) & \(n\) & \(m\) & \(n\) & \(m\) & \(n\) & \(n\) \\
				\(y\)                    & \(0\) & \(m\) & \(0\) & \(m\) & \(n\) & \(n\) & \(m\) & \(m\) & \(n\) & \(m\) & \(n\) & \(n\) & \(n\) \\
				\(b\vee (x\iimplies y)\) &
				\(n\)                    & \(1\) & \(m\) & \(m\) & \(1\) & \(1\) & \(1\) & \(m\) & \(1\) & \(n\) & \(1\) & \(1\) & \(1\)         \\
				                         &
				\(n\)                    & \(1\) & \(m\) & \(m\) & \(1\) & \(1\) & \(1\) & \(m\) & \(1\) & \(n\) & \(1\) & \(1\) & \(1\)         \\
			\end{tabular}
		\]
		\caption{Identity (5) of Lemma~\ref{Equations}; here the last line represents \((((b\lhd x) \iimplies (b\lhd y)) \iimplies (b \vee x)) \iimplies (((b\lhd y) \iimplies (b\lhd x))\iimplies (b \vee y))\)}\label{Table 3}
	\end{table}
	\begin{table}
		\[
			\begin{tabular}{c c c c c c c c c c c c c c}
				\(b\)                         &
				\(n\)                         & \(n\) & \(m\) & \(0\) & \(0\) & \(m\) & \(n\) & \(m\) & \(m\) & \(n\) & \(n\) & \(m\) & \(n\) \\[2pt]
				\(x\)                         &
				\(m\)                         & \(0\) & \(n\) & \(n\) & \(m\) & \(0\) & \(m\) & \(n\) & \(m\) & \(n\) & \(m\) & \(n\) & \(n\) \\[2pt]
				\(y\)                         &
				\(0\)                         & \(m\) & \(0\) & \(m\) & \(n\) & \(n\) & \(m\) & \(m\) & \(n\) & \(m\) & \(n\) & \(n\) & \(n\) \\[2pt]
				\((x\iimplies b)\iimplies y\) &
				\(0\)                         & \(m\) & \(0\) & \(1\) & \(1\) & \(n\) & \(m\) & \(1\) & \(n\) & \(m\) & \(n\) & \(1\) & \(n\) \\[2pt]
				                              &
				\(0\)                         & \(m\) & \(0\) & \(1\) & \(1\) & \(n\) & \(m\) & \(1\) & \(n\) & \(m\) & \(n\) & \(1\) & \(n\) \\
			\end{tabular}
		\]
		\caption{Identity (6) of Lemma~\ref{Equations}; here the last line represents \((((b \lhd  x) \iimplies (b \vee y))\iimplies (b \vee y)) \wedge ((b \vee y) \iimplies (b\lhd y))\)}\label{Table 4}
	\end{table}
\end{proof}

\bibliography{tim}{}
\bibliographystyle{amsplain-nodash}

\end{document}